\documentclass[10pt,a4paper,reqno]{amsart}
\usepackage{amsfonts,amsthm,latexsym,amsmath,amssymb,amscd,amsmath,epsf}
\usepackage{graphicx}
\pdfoutput=1

\newtheorem{theorem}{Theorem}
\newtheorem{lemma}{Lemma}
\newtheorem{corollary}{Corollary}
\newtheorem{proposition}{Proposition}
\newtheorem{remark} {Remark}

\newtheorem{problem} {Problem}

\newtheorem{conjecture}{Conjecture}

\newcommand{\al}{\alpha}

\newcommand {\bC} {\mathbb {C}}
\newcommand {\bR} {\mathbb R}

\newcommand {\Ga} {\Gamma}

\begin{document}
             \numberwithin{equation}{section}

             \title[Univalent disks of rational functions and Hermite-Biehler polynomials]
             {Maximal univalent disks of real rational functions and  Hermite-Biehler  polynomials}

\author[V.~Kostov]{Vladimir P. Kostov}
\address{Universit\'e de Nice, Laboratoire de Math\'ematiques,
Parc Valrose, 06108 Nice Cedex 2, France, {\it kostov@math.unice.fr}}

\author[B.~Shapiro]{Boris Shapiro}
\address{Department of Mathematics, Stockholm University, SE-106 91, Stockholm,
            Sweden, {\it shapiro@math.su.se}}

\author [M.~Tyaglov]{Mikhail Tyaglov}
\address{Institut f\"ur Mathematik, MA 4-5
Technische Universit\"at Berlin,
D-10623 Berlin, Germany, {\it tyaglov@math.tu-berlin.de}}

\date{\today}
\keywords{Hermite-Biehler theorem, root localization}
\subjclass[2000]{Primary 26C05, Secondary 30C15}

\begin{abstract} The well-known Hermite-Biehler theorem claims that a univariate monic polynomial $s$ of degree $k$ has all roots in the open upper half-plane if and only if  $s=p+iq$ where  $p$ and $q$ are real polynomials of degree $k$ and $k-1$ resp. with all real, simple and interlacing roots,  and $q$ has a negative leading coefficient.  Considering roots of $p$ as cyclically ordered on $\bR P^1$  we show that the open disk in $\bC P^1$ having a pair of consecutive roots of $p$ as its diameter is the maximal univalent disk for the function $R=\frac{q}{p}$.   This solves a special case of the so-called Hermite-Biehler problem.

\end{abstract}

\maketitle


\section{Introduction} \label{sec:int}  Rational functions of the form
\begin{equation}\label{eq:main}
R=\frac{q}{p},
\end{equation}³
where $p$ and $q$ have all real, simple and interlacing roots appear often in different areas of mathematics and enjoy a number of nice properties. Below we discuss  a (hopefully) new property of such $R$ related to the classical Hermite-Biehler theorem.  Namely, we start with  the following question.

\begin{problem}
For a given real rational function $F$ describe  real (i.e. invariant under the complex  conjugation) disks in $\bC P^1$ in which $F$ is univalent.
\end{problem}³

An obvious necessary condition for a real disk $D\subset \bC P^1$ to be a univalent domain of $F$ is that $F$ restricted to the real interval $D\cap \bR P^1$ is univalent.
The main result of this note  is that for rational functions of the form \eqref{eq:main} this simple necessary condition is, in fact, sufficient, see Theorem~\ref{th:main}.
It  has an immediate implication for the location of roots of  Hermite-Biehler polynomials, i.e. polynomials whose roots belong  to the open upper half-plane.  Denoting the monic polynomial proportional to the latter one by $s$ we get by the Hermite-Biehler theorem that $s=p+iq$ where $p$ and $q$ have all real, simple and  interlacing roots and different signs of their leading coefficients.  Since  the roots of  $s$ solve the equation $p+iq=0$ or, equivalently, $\frac{q}{p}=i$  we can apply  Theorem~\ref{th:main} to restrict the location of roots of $s$.  Additionally, our result gives a partial answer to the following  general Hermite-Biehler problem  as formulated in \cite{Fi}, p.~575.\footnote{When preparing this paper we received a sad news that  Professor S.~Fisk passed away after a long period of   illness.}

\medskip
\noindent
{\bf Question.}³ Given a pair of real polynomials $(p,q)$ give restrictions on the location of the roots of $p+iq$ in terms of the location of the roots of $p$ and $q$.

\medskip In what follows we will assume for simplicity that $p$ and $q$ are monic polynomials which does not affect any of the results below.
We start with a certain localization result for the critical points of $R$ given by  \eqref{eq:main}.
For a given  pair $(p,q)$ of polynomials define its {\it  Wronskian}  as
$$W(p,q)=p'q-q'p.$$
 Properties of the Wronski map were studied in details in e.g. \cite{EG1} and \cite{EG2}.
 Notice if $p$ and $q$ are coprime then  roots of $W(p,q)$ are exactly the critical points of $R=\frac{q}{p}$. (Here $R$ is considered as a map of $\bC P^1$ to itself.) In the  case  when $p$ and $q$ have all  real, simple  and interlacing roots  one can easily see that all roots of $W(p,q)$ are non-real.  (The latter fact  was already known  to C.~F.~Gau\ss{}.)

Our first result concerns the location of these roots.  Assume that $p$ and $q$ are real polynomials of degrees $k$ and $k-1$ resp. with all real, simple and interlacing roots and  let $p_1<p_2<...<p_{k}$ be the roots of $p$. Denote by $D_j, \;j=1,...,k-1$ the open disk bounded by the circle $C_j\subset \bC $ having the pair $(p_j,p_{j+1})$ as its diameter and denote by $D_0$ the open disk  bounded by the circle $C_0\subset \bC $ having the pair $(p_1,p_{k})$ as its diameter.  Set $\Omega_p=\bar D_0\setminus \bigcup_{j=1}^{k-1}D_j$ where $\bar D_0$ is the closure of $D_0$, see Fig.~1. Our first result is as follows.

\begin{theorem} \label{th:wron}  Under the above assumptions all roots of $W(p,q)$ lie in $\Omega_p$.
\end{theorem}

 Similar statement for the case $q=p'$ can be found in \cite{DS}. Further, note that for any real $\alpha$, the roots of the polynomial $p+\alpha q$ are real and interlacing the roots  of~$p$.
 Denoting by $p_j(\alpha)$ the roots of the polynomial $p+\alpha q$ consider the open disks $D_j(\al),\;j=1,...,k-1$ having the intervals $(p_j(\al),p_{j+1}(\al))$ as their diameters and denote by $D_0(\al)$ the open disk having $(p_1(\al),p_k(\al))$ as its diameter. Finally, let $\Omega_p(\al)=\bar D_0(\al)\setminus \bigcup_{j=1}^{k-1}D_j(\al).$  Since for any real $\al$ one has $W(p+\al q,q)=W(p,q)$, Theorem~\ref{th:wron}³implies the following.

\begin{corollary}³ All roots of $W(p,q)$ lie in  $\bigcap_{\al\in \bR}\Omega_p(\al)$.
\end{corollary}³


 Since in our case all roots of $W(p,q)$ are non-real it is natural to ask if every real polynomial of  even degree with all non-real roots can be represented as the Wronskian of a suitable pair $(p,q)$ as above.  
 Consider the space $Int_k$ of all pairs $(p,q)$ where $p=z^{k}+a_1z^{k-2}+...+a_{k-1}$, $q=z^{k-1}+b_1z^{k-2}+...+b_{k-1}$ with all real, simple and interlacing roots. The next statement  follows from the main result of \cite{Za} which was conjectured by J.~Milnor.

\begin{proposition}\label{pr:wronsk} The Wronski map $W$ defines a diffeomorphism of the space $Int_{k-1}$ and the space
 $Pol_{2k-2}$ consisting of polynomials $u(z)=z^{2k-2}+c_1z^{2k-3}+...+c_{2k-2}$ which are strictly positive on the whole real axis $\bR$, i.e. it is a bijective map between $Int_{k-1}$ and $Pol_{2k-2}$ (with everywhere non-vanishing Jacobian).
\end{proposition}³

\begin{figure}

\begin{center}
\includegraphics[scale=0.35]{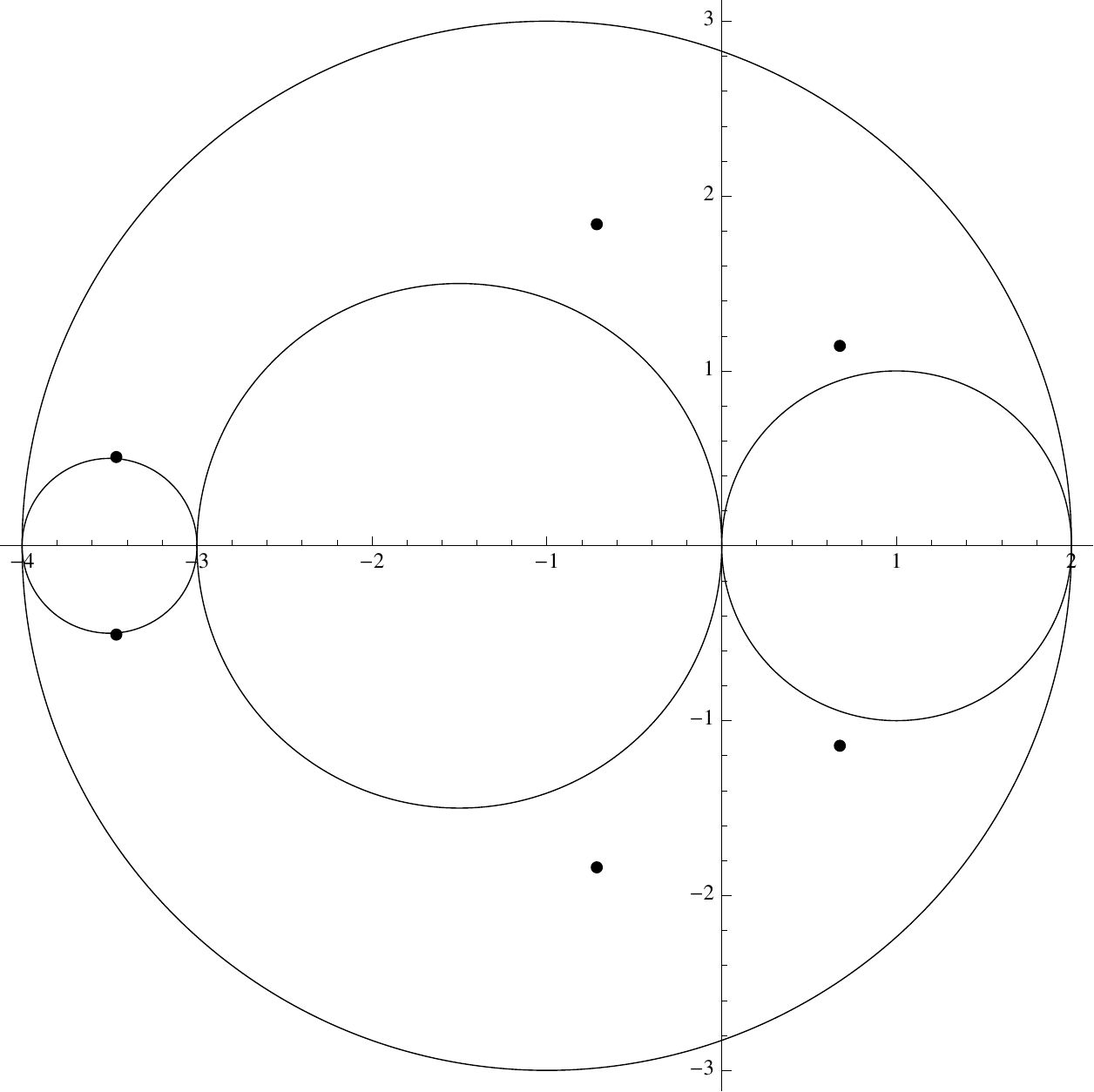} \includegraphics[scale=0.35]{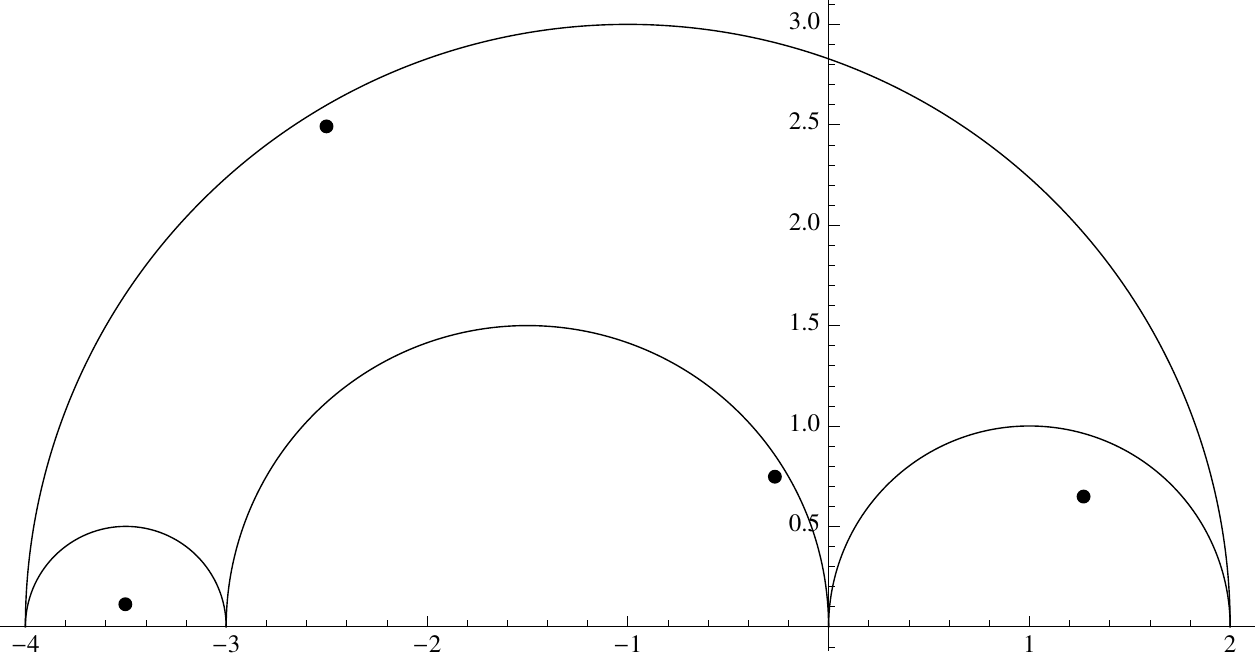}
\end{center}

\caption{Critical points of $R=\frac{q}{p}$ and roots of $s=p-4 iq$ for  $p=(z+4)(z+3)z(z-2)$ and $q=(z+7/2)(z+1)(z-1)$.}
\label{fig1}
\end{figure}

Finally, the main result of this note is as follows. Take $R=\frac{q}{p}$ as in \eqref{eq:main}.

\begin{theorem}\label{th:main} In the above notation  the open disk $D_j(\al)$ is a maximal real univalent disk of $R$.
\end{theorem}






Theory of univalent functions is a very classical domain closely related to the theory  of quasiconformal mappings and Teichm\"uller spaces, see e.g. \cite {Leh}.  One of the central results often called the criterion of univalency of functions in a disk was proved by Z.~Nehari in 1949. It  claims that if the quotient of the absolute value of the Schwarzian derivative of $F$ by the square of the appropriate hyperbolic metrics in a  disk is at most $2$, then $F$ is univalent.  This result is also best possible within  the whole class of holomorphic  functions as shown by E.~Hille but it does not have to be sharp for individual functions. It is very tempting to understand the relation between Nehari's criterion and the above Theorem~\ref{th:main}.

\medskip
\noindent
{\it Acknowledgements.}³ The first and the third authors want to acknowledge the hospitality    of   the  Mathematics department of     Stockholm University in October 2009 and January 2010. The authors want to thank  A.~Eremenko  for relevant explanations of the main result of \cite{Za} and P.~Br\"anden for important discussions of the topic.

\section{Proofs}\label{sec:pr}

We first prove Theorem~\ref{th:wron}.
\begin{proof}
Denote by  $p_1<q_1<p_2<q_2<...<q_{k-1}<p_k$  the roots of $p$ and $q$ respectively.  Since for any non-degenerate matrix
$\begin{pmatrix} a & b\\ c& d \end{pmatrix}$ one has
\begin{equation}\label{shifted.Wronskian}
W(p,q)=(ad-bc)\cdot W(ap+bq,cp+dq),
\end{equation}
 we can w.l.o.g. assume that $p=\prod_{j=1}^k(z-p_j)$ and $q=\prod_{j=1}^{k-1}(z-q_j)$.

Notice that the function $ R=\frac{q}{p}$ can be represented
as
\begin{equation*}
 R(z)=\sum\limits_{j=1}^{k}\dfrac{\alpha_j}{z-p_j},
%
\text{    where   }
%
\alpha_j=\dfrac{q(p_j)}{p'(p_j)}>0.
\end{equation*}
As we mentioned above  the set of all roots of $W(p,q)$ coincides with that of  the derivative
\begin{equation*}
 R^\prime(z)=-\sum\limits_{j=1}^{k}\dfrac{\alpha_j}{(z-p_j)^2}.
\end{equation*}

The following lemma is a modification
of~\cite[Theorem~2.5]{DS} and implies Theorem~\ref{th:wron}.
\begin{lemma}\label{lem.1}
All roots of $W(p,q)$ lie inside or on the above
circle  $C_0$.
\end{lemma}
\begin{proof}
Note that w.l.o.g. we can  assume that
$p_{1}+p_{k}=0$, since the statement is invariant under the addition of a real constant to the independent variable $z$.
Assuming that $p_{1}+p_{k}=0$  set
$\theta_j=\arg(z-p_j)$, $j=1,\ldots,k$. Since $p$ and $q$ are real polynomials,  we can restrict our attention to the
case where $z$ lies in the upper half-plane and, therefore,
\begin{equation}\label{lem.1.proof.1}
0<\theta_1<\theta_2<\ldots<\theta_{k}<\pi.
\end{equation}
Now suppose that $|z|>\dfrac{p_{k}-p_1}{2}$ which implies that  the
origin lies outside the circle of radius
$\dfrac{p_{k}-p_1}{2}$ centered at $z$. Therefore,
$\theta_{k}-\theta_1<\pi/2$, by the cosine rule. Since
$\arg(z-p_j)^{-2}=-2\theta_j$, then  by~\eqref{lem.1.proof.1}
we get
\begin{equation*}
-2\theta_k\leqslant\arg(z-p_j)^{-2}\leqslant-2\theta_1,\quad
j=1,\ldots,k,
\end{equation*}
and, therefore, since all $\alpha_j>0$
\begin{equation*}
2\theta_1\leqslant\arg \frac{-{\alpha_j}}{(z-p_j)^{2}}\leqslant 2\theta_k,\quad
j=1,\ldots,k.
\end{equation*}

Thus,  since $|2\theta_k-2\theta_1|<\pi$ one has that
all the numbers $\
\dfrac{-\alpha_j}{(z-p_j)^2}\ ,\; j=1,...,k$ lie  on one side w.r.t.  the straight
line through the origin with the slope $2\theta_1$. Since $ R^\prime(z)$ is the  sum of these
numbers, then $ R^\prime(z)\neq 0$ at least for $|z|>\dfrac{p_{k}-p_1}{2}$.
\end{proof}

Finally,  notice that the assertion  of Lemma~\ref{lem.1}  is invariant under M\"obius transformations of $\bC P^1$, since  we can transform any pair  of consecutive roots $(p_j,p_{j+1})$ into the pair $(\tilde p_1,\tilde p_k)$ of the smallest and largest roots of another polynomial with the same properties by applying a suitable M\"obius transformation. By \eqref{shifted.Wronskian}  the roots of  $W(p,q)$ do not change under any non-degenerate M\"obius transformation and the claim follows.  \end{proof}

%
%
Denoting  by $p_j(\alpha)$ the roots of the polynomial
$p+\alpha q$, we  get
$\lim_{\alpha\to +\infty}p_j(\alpha)= q_j$ for $j=1,\ldots,k-1,$ and
$\lim_{\alpha\to +\infty}p_k(\alpha)=+\infty$. Analogously,
$\lim_{\alpha\to -\infty}p_j(\alpha)= q_{j-1}$ for $j=2,\ldots,k,$ and
$\lim_{\alpha\to -\infty}p_1(\alpha)=-\infty$. Lemma~\ref{lem.1} together with
\eqref{shifted.Wronskian} imply  that all zeros of
Wronskian $W(p,q)$ lie in the intersection of all circles centered
at $\dfrac{p_k(\alpha)+p_1(\alpha)}2$ whose radii are equal to
$\dfrac{p_k(\alpha)-p_1(\alpha)}2$.
\medskip
\begin{corollary}
All roots of the Wronskian $W(p,q)$ lie in the intersection of the
closed disk centered at $\dfrac{p_{k}+p_{1}}{2}$ of
radius $\dfrac{p_{k}-p_1}{2}$ and the strip: $q_1\leqslant\text{Im } z\leqslant q_{k-1}$.
\end{corollary}

\medskip
To prove Proposition~\ref{pr:wronsk}³ let us first quote the main result of \cite{Za}.

\begin{theorem}³\label{th:za}
Let $c_1, ... , c_d$ be a sequence of  (not necessarily distinct) points in the open unit disk
$D$ in the complex plane. Then there exists a unique proper holomorphic map $f : D \to D$ of
degree $d + 1$ normalized as $f(0) = 0$ and $f(1) = 1$ whose critical points are exactly $c_1, ... , c_d$.
(If a certain point is repeated several times among $c_1, ... , c_d$ this means that the corresponding critical value has the corresponding multiplicity.)
\end{theorem}³

To finish the proof of  Proposition~\ref{pr:wronsk} note  that a real rational function $R=\frac{q}{p}$ maps the open upper half-plane on the open upper half-plane if and only if
 $p$ and $q$ have all real, simple and  interlacing roots and  opposite signs of their leading coefficients.  Since the upper half-plane can be mapped to the unit disk by a M\"obius transformation we can arbitrarily  prescribe $k-1$ critical points in the upper half-plane and find a degree $k$ map of the upper half-plane to itself with these (and no other) prescribed critical points. This map will be uniquely defined by Theorem~\ref{th:za} if we additionally use the restrictions as in Proposition~\ref{pr:wronsk}. \qed

 \medskip
The next two propositions give interesting information on the behavior of $R$ on and inside the circle $C_j$  and (in our opinion) are of independent interest.

 \begin{proposition}\label{pr:convex} In the above notation and for any $j=0,...,k-1$ the function $| R|^2(z)$ restricted to the  half-circle $C_j^+$ is convex as a function of the real part of $z$ where $C_j^+$ is the part of $C_j$ lying in the upper half-plane. In other words, $\Phi(x)=|R|^2(x+i\sqrt{r_j^2-(x-\rho_j)^2})$ where $x\in [p_j,p_{j+1}],\; r_j=\frac{p_{j+1}-p_j}{2},\; \rho_j=\frac{p_{j+1}+p_j}{2}$ is convex as a function of $x$.
 \end{proposition}³





 \begin{figure}

\begin{center}
\includegraphics[scale=0.65]{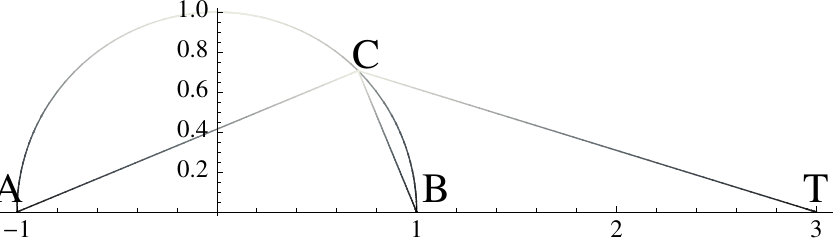}
\end{center}

\vskip 0.3cm

\caption {Illustration to the proof of Lemma~\ref{lm:stew}.}
\label{fig3}
\end{figure}

We will first prove   convexity of  $| R|^2$  w.r.t.  a more convenient  coordinate on $C_j^+$.
Denote temporarily by $A$ and  $B$ the roots $p_j$ and $p_{j+1}$ resp. and denote by $C$  an arbitrary point of the half-circle $C_j^+$.
Let  $w:=|BC|^2$ be the square of the length of the segment $BC$. Obviously, $w$ varies from $0$ to $(p_{j+1}-p_j)^2$ when $C$ runs along $C_j^+$ from $p_{j+1}$ to $p_j$. Let us  express
the function $|{R}|^2$ in terms of $w$.

\begin{lemma}\label{lm:stew} One has
\begin{equation}\label{eq:RR}
| R|^2(w)=\frac{\prod_{l=1}^{k-1}( s_lw+t_l)}{\prod_{l=1}^k(u_lw+v_l)},
\end{equation}
where $t_l=(q_l-p_{j+1})^2,\; v_l=(p_l-p_{j+1})^2,\; s_l=\frac{2q_l-p_j-p_{j+1}}{p_{j+1}-p_j},\;v_l=\frac{2p_l-p_j-p_{j+1}}{p_{j+1}-p_j}.$

\end{lemma}

 \begin{proof}

Take an arbitrary point $T$ on the real axis (assume for the moment that $T>p_{j+1}$) and consider the triagle $ACT$, see Fig.~\ref{fig3}.
By the well-known Stewart's theorem from elementary geometry applied to the triangle $ACT$ with a cevian $BC$ one gets

$$|BC|^2=\frac{|AB|}{|AT|} |CT|^2+\frac{|BT|}{|AT|} |AC|^2-|AB||BT|~.$$
Since $|AC|^2=|AB|^2-|BC|^2$ the latter formula is equivalent to
\begin{equation}\label{eq:ste}
|CT|^2=\frac{2T-A-B}{B-A}|BC|^2+(B-T)^2=\frac{(2T-A-B)z}{B-A}+(B-T)^2.
\end{equation}
One can easily check that \eqref{eq:ste} is valid even if  $T\le p_{j+1}$ which  implies \eqref{eq:RR}.\end{proof}³

\begin{remark}
Notice that although we are primarily interested in  $| R|^2(w)$ when $w$ belongs to the interval
$(0,(p_{j+1}-p_j)^2)$ the formula \eqref{eq:RR} defines
 $| R|^2(w)$ globally as a rational function of $w$ which we will be using below.
 \end{remark}

\begin{proof} [Proof of Proposition~\ref{pr:convex}] Invariance under an affine change of $z$ allows us  w.l.o.g. to assume  that $p_j=-1$ and $p_{j+1}=1$ which leads to the formula:
\begin{equation}\label{eq:fin}
| R|^2(w)=\frac{\prod_{l=1}^{k-1}(q_lw+(1-q_l)^2)}{\prod_{l=1}^{k}(p_lw+(1-p_l)^2)}.
\end{equation}

If we denote by $\xi$ any of the points $q_l$'s (resp. $p_l$'s), then the corresponding linear factor in the numerator
(resp.  denominator)  has the form  $\xi w+(1-\xi)^2$.  Its only root is given by the formula $\zeta=\Psi(\xi)=-\frac{(1-\xi)^2}{\xi}$. The rational map
$\Psi(\xi)$ is double covering of $\bR P^1\setminus (0,4)$ by $\bR P^1$, namely, the interval $(+0,1]$ is mapped to $(-\infty,0]$,  the interval $[1,+\infty)$ is mapped to $[0,-\infty)$, the interval $(-\infty,-1]$ is mapped to $(+\infty, 4]$, and, finally the interval $[-1,-0)$ is mapped to $[4,+\infty)$, see Fig.~\ref{fig4}.  This observation implies that if we consider the images of our roots $p_1<q_1<....<p_j=-1<q_j<p_{j+1}=1<...<q_{k-1}<p_k$ under the map $\Psi$, then $p_{j+1}$ will be mapped to $0,$ and $q_{j+1}<...<p_k$ will be mapped to the negative half-axis in the order-reversing manner. Analogously, $p_j$ is mapped to $4,$ and all roots $p_1<q_1<...<q_{j-1}$ will be mapped to $(4,\infty)$ in the order-reversing manner.
The only root $q_j\in(-1,1)$ can be mapped completely arbitrarily to $\bR P^1\setminus [0,4]$. In particular, if $q_j=0$, then $\Psi(q_j)=\infty$. Thus, considering $| R|^2(w)$  as a function on $\bR P^1$ we obtain that exactly $(k-2)$ of totally $k$ intervals between  consecutive roots of the denominator of  \eqref{eq:fin}  contain exactly one root of the numerator.  The remaining two intervals will be called {\it exceptional}. One of these exceptional intervals is $(0,4)$, where there are no roots of the numerator. The other exceptional interval is the one containing   $\Psi(q_j)$. Besides $\Psi(q_j)$ it contains one more root of the numerator  (counting multiplicities) as shown on Fig.~\ref{fig4}. (Recall that the numerator and denominator of \eqref{eq:fin}  are to be considered as homogeneous polynomials of degree $k$ in homogeneous coordinates on $\bR P^1$, i.e. the numerator has always a root at $\infty$.)
Notice that it might happen that $\Psi(q_j)=\Psi(p_i),\; (i\neq j, j+1)$. In this case $\deg |R|^2(w)$ decreases by $1$ but the required statement still holds.

The function $| R|^2(w)$ restricted to each closed non-exceptional interval on $\bR P^1$ maps it bijectively to  the whole $\bR P^1$. Consider
now the behavior of $| R|^2(w)$ on $(0,4)$. It is strictly positive there and tends to $+\infty$ when $w\to 0+$ or $w\to 4-$. One can immediately see that  there can be at most two solutions to the equation $ | R|^2(w)=aw+b$ on $(0,4)$, where $a, b$ are arbitrary  real numbers. Indeed, consider the equation   $ | R|^2(w)=aw+b$ globally on $\bR P^1$. Notice that unless $\Psi(q_j)=\Psi(p_i),\; (i\neq j, j+1)$ the degree of $| R|^2(w)$ as a rational function equals $k$ and, therefore, the latter equation can have at most $k$ real solutions on the whole $\bR P^1$ counting multiplicities. We already know that under the same assumption this equation has exactly one real solution on each of  the $(k-2)$ non-exceptional intervals. Thus, there can be at most $2$ solutions on $(0,4)$, see Fig.~\ref{fig4}. Finally, notice that a  smooth positive function defined on $(0,4)$, tending to $+\infty$ on both ends and whose graph intersects  any straight line at most twice (counting multiplicities) is convex.  Thus,  we have proved that $|R|^2(w)$ is a convex function of $w\in (0,4).$  (The exceptional case $\Psi(q_j)=\Psi(p_i),\; (i\neq j, j+1)$ is considered  in exactly the same manner.)  To prove that $|R|^2(x)$ is convex where $x$ is the real part of $z$ on the upper half-circle of the unit circle notice that $w=2(1-x)$. The latter relation is easily obtained from  the Pythagorean theorem. Since convexity is preserved under affine changes of coordinates the result follows. \end{proof}³

  \begin{figure}

\begin{center}
\includegraphics[scale=0.5]{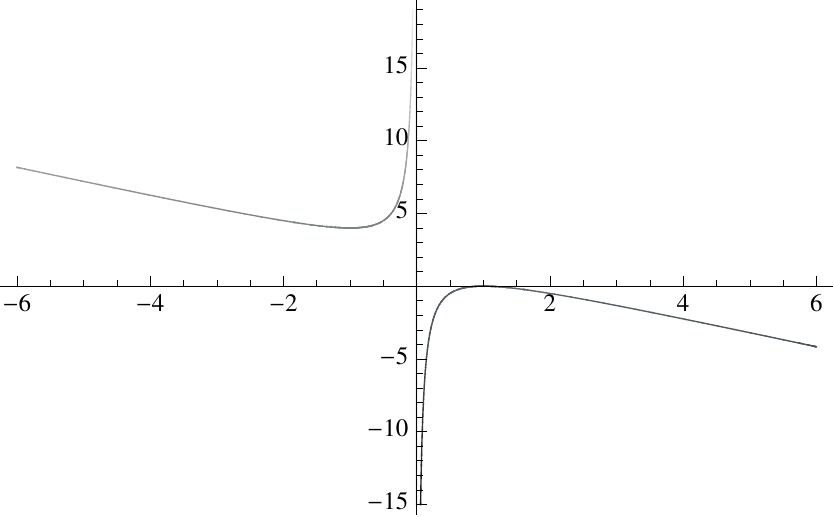} \quad \includegraphics[scale=0.48]{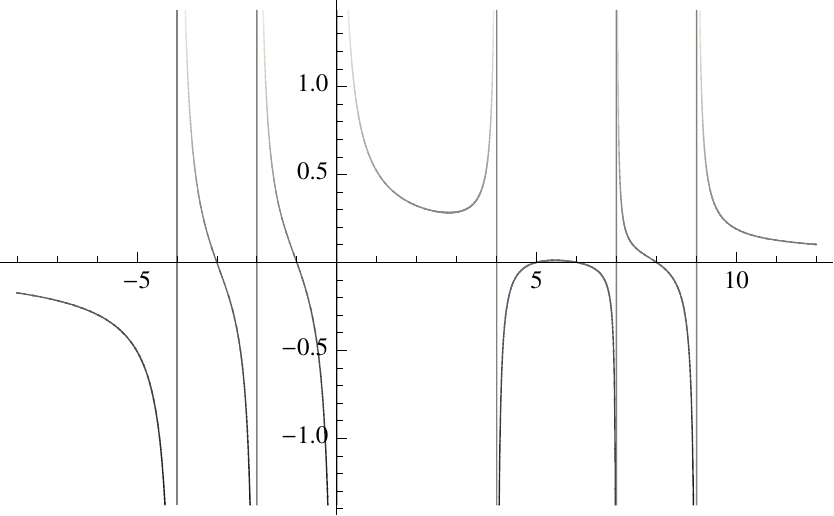}
\end{center}

\vskip 0.3cm

\caption {$\Psi(\xi)=-\frac{(1-\xi)^2}{\xi}$ and  an example of $| R|^2(w),\;w\in \bR$. }
\label{fig4}
\end{figure}

Now  consider the family of level curves  $\Ga_r:\{| R|^2(z)=r^2 \ge 0\}$ and intersect all curves in this family with a given disk $D_j$, see Fig.~\ref{fig7}. Call a positive number  $r$  ³{\it non-critical} for $R(z)$ if it is not the  modulus of some of its critical values.  Notice that if $r>0$ is non-critical,  then the curve
 $\Ga_r$ consists of a number of  smooth ovals and it is invariant under the complex conjugation. Set $m=\min_{z\in C_j^+} |R|(z)$.
Notice that by Proposition~\ref{pr:convex}  the value $m$ is attained at the unique point $P$ on $C_j^+$.

 \begin{proposition}\label{pr:imp}  For any $0<r\le m$ the intersection of the level curve $\Ga_r: | R|^2(z)=r^2> 0$ with $\bar D_j$ is a single oval containing $q_j$, where $\bar D_j$ is the closure of $D_j$. For any $r>m$  this intersection is a pair of arcs each crossing the interval $(p_j,p_{j+1})\subset \bR$.  
 \end{proposition}

 \begin{proof}
 Let us first show that for any non-critical $r$ the curve $\Ga_r\subset \bC P^1$ consists of exactly $k$ smooth ovals each of which intersects $\bR$ in two points and is, therefore, invariant under the complex conjugation. Notice that the function $R(z)$ induces  a  $k$-fold covering of the circle of radius $r$  by $\Ga_r$. For a non-critical $r$ the curve $\Ga_r$ consists of at most $k$ smooth ovals and is invariant under conjugation. Each such oval is either self-conjugate and, thus intersecting $\bR$ or does not intersect $\bR$ and, then it has a complex conjugate pair. Observe that the function $|R|^2(z)$ restricted to $\bR$ coincides with $R^2(z)$ and is  a real non-negative rational function of degree $2k$.
 Since $|R|^2(z)$ vanishes at each $q_j\in (p_j,p_{j+1})$ and tends to $+\infty$ when $z\to p_j$ one can easily observe that it is convex on each $(p_j,p_{j+1})$ implying that for any $r>0$ the curve $\Ga_r$ intersects $\bR$ exactly $2k$ times and, thus, consists of $k$ self-conjugate ovals for any non-critical $r$.

Notice that singularities of $\Ga_r,\, r>0$ can only occur at the critical points of $R(z)$. Thus,  since $D_j$ contains no critical points of $R(z)$,  any curve  $\Ga_r,\, r>0$ restricted to  $D_j$ is non-singular. The function $|R|^2(z)$ is non-negative in $\bC$ and vanishes only at the roots $q_1<...<q_{k-1}$. Thus, it vanishes only once inside $D_j$, namely, at $q_j$. Therefore, for sufficiently small $r>0$ the curve $\Ga_r$ intersected with  $D_j$ is a small oval around $q_j$.  Again due to the absence of critical points inside $D_j$ this family of ovals persists until $r$ reaches the value $m$. The oval $O_m$ corresponding to $r=m$ will be tangent to the circle $C_j$, see Fig.~\ref{fig7}. Indeed, since $m=\min|R|^2(z)$ on $C_j^+$ the oval $O_m$ has to have a common point $P$ with $C_j^+$. The latter point $P$ is unique due to convexity  of  $|R|^2(z)$ on $C_j^+$ w.r.t. the real part of $z$ and since $O_m$ is smooth $P$ is the tangency point of $O_m$ and $C_j^+$. The remaining part of $D_j$ consists of two crescents each of  which will be covered by a family of arcs formed by intersecting $\Ga_r, r>m$ with $D_j$. Indeed,  denote the intersection points of $O_m$ with $\bR$ as $a<b$. Note that $|R(z)|$ restricted to $C_j^+$ is strictly monotone decreasing from $+\infty$ to $m$ ³and then strictly monotone increasing back to $+\infty$ when $z$ moves along $C_j^+$ from $p_j$ to $p_{j+1}$. Thus we get that for each $r>m$ there exist a unique point in $(p_j, a)$  and a unique point on the part on $C_j^+$ from $p_j$ to $P$ such that $|R(z)|=r$. A similar pair of points exists on $(b,p_{j+1})$ and the part of $C_j^+$ from $P$ to $p_{j+1}$ resp.   The curve $\Ga_r$ must necessarily connect the first pair of points and the second pair of points. Indeed, if we follow $\Ga_r$ from the point of $\bR$, then it can  not stay within $D_j$ since it can not return back to the real axis. Moreover,  it can only leave $D_j$ at the unique point on $C_j^+$ due to convexity of $|R|^2(z)$ on $C_j^+$ w.r.t the real part of $z$.    \end{proof} 



  \begin{figure}

\begin{center}
\includegraphics[scale=0.5]{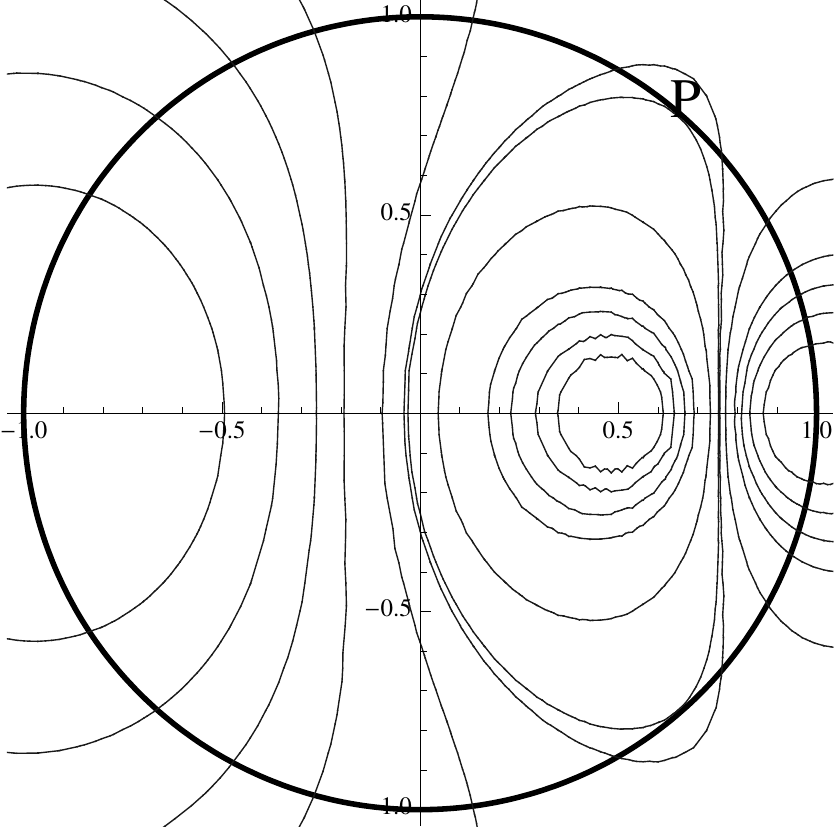}
\end{center}

\vskip 0.3cm

\caption {Level curves of $| R|^2(z)$  within the disk $|z|\le 1$ where  $R(z)=\frac{(z-1/2)(z-2)(z+2)}{(z-1)(z+1)(z-3)(z+3)}$.}
\label{fig7}
\end{figure}

Let us finally settle Theorem~\ref{th:main}.

\begin{proof} A maximal open contractible univalent  domain $D\subset \bC P^1$ of a holomorphic function $F$ has the property that  $F$ is injective on $D$ but  some points on  the boundary of $D$ have the same image. In our case the image under $R$ of the real interval $[p_j,p_{j+1}]$ is the whole $\bR P^1$ with $\infty$ covered twice by $p_j$ and $p_{j+1}$. Analogously, the image under $R$ of the half-circle $C_j^+$ is a closed curve starting and ending at $\infty$. If we show that $R(C_j^+)$  is a simple loop in $\bC P^1$ non-intersecting $\bR P^1$ at points other  than $\infty$, then Theorem~\ref{th:main} will easily follow. Indeed, in this case the open half-disk $D_j^+$ bounded by $[p_j,p_{j+1}]$ and $C_j^+$ will be mapped bijectively to the part of the lower half-plane $\Im z <0$ lying above the loop $R(C_j^+)$. Analogously, the open half-disk $D_j^-$ bounded by $[p_j,p_{j+1}]$ and $C_j^-$ will be mapped bijectively to the part of the upper half-plane $\Im z >0$ lying below the loop $R(C_j^-)$. (Indeed, our domain is a disk which is mapped locally diffeomorphically on its image and its boundary is mapped onto the union of 2 closed smooth loops having a single common point.)  The fact that  $R(C_j^+)$  is a simple loop is almost true except for the special case when $\deg R=2$. Notice that  any function $R$ as in
\eqref{eq:main}  restricted to  $\bR P^1$  induces a $k$-fold covering of $\bR P^1$ by $\bR P^1$. Since $\deg R=k$ there is no points in $\bC P^1\setminus \bR P^1$ whose image under $R$ is real. Therefore, no points from $C_j^+$ except for its endpoints are mapped to $\bR P^1$. To prove that
$R(C_j^+)$  is a  simple loop  for $\deg R>2$ notice that our statement is invariant under the action of the M\"obius group on $R$ and we can w.l.o.g. assume that $p_1=-1<q_1=0<p_2=1<q_2<p_3<...<q_{k-1}<p_k$.

\begin{lemma}\label{lm:final} Under the above assumptions and if $\deg R(z)>2$, then the function  $\Re R(z)$ is a monotone decreasing when restricted to  $C_1^+$ parameterized from $p_1$ to $p_2$. In the special case $\deg R(z)=2$ the function $\Re R(z)$  is constant.  (Here $\Re F$ means the real part of $F$.)
\end{lemma}
\begin{proof} Let us express $\Re R(x)$ explicitly.  Recall that
$$R(z)=\sum_{j=1}^k\frac{\al_j}{z-p_j}=\frac{\al_1}{z+1}+\frac{\al_2}{z-1}+\sum_{j=3}^k\frac{\al_j}{z-p_j},$$
where $\al_j>0$ and $p_j>1$ for $ j>2$. Then substituting  $z=x+i y$ one gets
$$\Re R(x+iy)=\frac{\al_1(x+1)}{(x+1)^2+y^2}+\frac{\al_2(x-1)}{(x-1)^2+y^2}+\sum_{j=3}^k\frac{\al_j(x-p_j)}{(x-p_j)^2+y^2}.$$
To show that $\Re R$ is monotone decreasing on the upper half of the unit circle let us parameterize it by $x\in[-1,1]; \; y=\sqrt{1-x^2}$. Then
$$\Re R(x)=\frac{\al_1(x+1)}{(x+1)^2+(1-x^2)}+\frac{\al_2(x-1)}{(x-1)^2+(1-x^2)}+\sum_{j=3}^k\frac{\al_j(x-p_j)}{(x-p_j)^2+(1-x^2)}=$$
$$=\frac{\al_1}{2}-\frac{\al_2}{2}+\sum_{j=3}^k\frac{\al_j(x-p_j)}{(x-p_j)^2+(1-x^2)}.$$
Notice that each summand $\Psi_j(x)=\frac{\al_j(x-p_j)}{(x-p_j)^2+(1-x^2)}$ is a monotone decreasing function for $\al_j>0,\; p_j>1, \;x\in[-1,1]$.  Indeed, simple calculation shows that
$$\Psi'_j(x)=\al_j\frac{1-p_j}{(1+p_j^2-2p_jx)^2}<0.$$ \end{proof}³
Thus, we conclude that since $\Re R$ is a monotone decreasing function the curve $R(C_j^+)$ is a simple loop starting and ending at $\infty$ and not intersecting $\bR P^1$. Theorem~\ref{th:main} is settled for $\deg R>2$.  In case $\deg R=2$ one can easily see  that $R(C_j^+)$ is a double covering of a vertical ray in $\bC$ lying below the real axis.
\end{proof}

\begin{corollary} The number $m=\min_{z\in C_j^+} |R|(z)$ is smaller than the least modulus among the critical values of $R(z)$.

\end{corollary}

Finally, let us come back to the by Hermite-Biehler theorem
stating, for given monic real polynomials $p$ and $q$ of degrees
$k$ and $k-1$, respectively, that the polynomial
\begin{equation*}
s(z)=p(z)+i\alpha q(z)
\end{equation*}
has all zeroes in the upper half-plane if and only if $\alpha<0$
and the polynomials have real, simple and interlacing zeroes. More
generally, by Hermite-Biehler theorem, all roots of the equation
\begin{equation}\label{eq.1}
\dfrac{q(z)}{p(z)}=re^{i\varphi},\qquad
r>0,\,\,\varphi\in(-\pi,\pi].
\end{equation}
lie in the upper (resp. lower) half-plane if and only if the
polynomials have real, simple and interlacing zeroes and
$\varphi\in(-\pi,0)$ (resp. $\varphi\in(0,\pi)$). All roots of the
equation~\eqref{eq.1} are real and interlaces both zeroes of the
polynomials $p$ and $q$ if and only if $\varphi$ equals $0$ or
$\pi$.

Theorem~\ref{th:main} allows us to locate the roots of the
equation~\eqref{eq.1} more precisely and, therefore, to improve the
Hermite-Biehler theorem. In fact, Theorem~\ref{th:main} implies the following
corollary.

\begin{corollary}\label{corollary.1}
Each open disk $D_j$, $j=1,\ldots,k$, as well as the open disk
$\mathbb{C}\setminus \overline{D}_0$ on $\mathbb{C}P^1$ contains
at most one root of the equation~\eqref{eq.1}.
\end{corollary}

Considering the equation~\eqref{eq.1} as an equation on
$\mathbb{C}$, we obtain from Corollary~\ref{corollary.1} the
following statement.

\begin{corollary}\label{corollary.2}
All solutions of the equation~\eqref{eq.1} except at most one lie
in the closed disk $\overline{D}_0$.
\end{corollary}









\section{Some open problems}³

 \noindent
 1. Describe the boundary of $\bigcap_{\al \in \bR}\Omega_\al$ and  find its algebraic equation.

 \noindent
 2.  Is Nehari's criterion (see introduction) satisfied  for disks $D_j(\al)$?

 \noindent
 3. Are there other classes of real rational functions for which the univalency of restriction guarantees the univalency in the disk? In particular, what happens to real rational functions with all real roots and poles?

 \noindent
4. Find an analog of Theorem~\ref{th:main} for pairs of real-rooted polynomials $(p,q)$ satisfying the assumptions of Dedieu's theorem, see \cite{De} in which case the straight segment connecting $p$ and $q$ consists of real-rooted polynomials.

 \noindent
 5. Conjecture. For any $R$ as in \eqref{eq:main} the loop $R(C_j^+)$ is a convex curve on $\bC$.

\end{document}